\documentclass[20pt]{amsart}
\usepackage{amscd,amssymb}
\usepackage{amsthm,amsmath,amssymb}
\usepackage[matrix,arrow]{xy}

\newtheorem{theorem}[equation]{Theorem}
\newtheorem{proposition}[equation]{Proposition}
\newtheorem{lemma}[equation]{Lemma}
\newtheorem{corollary}[equation]{Corollary}
\newtheorem{conjecture}[equation]{Conjecture}

\theoremstyle{definition}
\newtheorem{example}[equation]{Example}

\theoremstyle{remark}

\makeatletter\@addtoreset{equation}{section}\makeatother

\begin{document}

\large

\title{One base point free theorem for weak log Fano threefolds}

\thanks{The work was partially supported by
RFFI grant No. 08-01-00395-a and grant N.Sh.-1987.2008.1.}

\sloppy

\author{Ilya Karzhemanov}

\maketitle

\begin{abstract}
Let $(X,D)$ be log canonical pair such $\dim X = 3$ and the
divisor $-(K_X + D)$ is nef and big. For a special class of such
$(X,D)$'s we prove that the linear system $|-n(K_{X}+D)|$ is free
for $n \gg 0$.
\end{abstract}

\bigskip

\section{Introduction}
\label{section:introduction}

Let $X$ be algebraic variety\footnote{All algebraic varieties are
assumed to be projective and defined over $\mathbb{C}$.} of
dimension $\geqslant 2$ with a $\mathbb{Q}$-boundary $D$ such that
the pair $(X, D)$ is log canonical and the divisor $-(K_{X}+D)$ is
nef and big. Then one may consider the following

\begin{conjecture}[M. Reid (see \cite{kawamata-crep-blowup}, \cite{shok-complements})]
\label{theorem:reid-conjecture} The linear system $|-n(K_{X}+D)|$
is free for $n \gg 0$.
\end{conjecture}

According to \cite[Proposition 11.1]{prok-complements} (see also
\cite{shok-complements}), Conjecture~\ref{theorem:reid-conjecture}
is true when $\dim X = 2$. Unfortunately, it is false when $\dim X
\geqslant 3$:

\begin{example}[see \cite{gongyo}]
\label{example:counter-example} Let $Z$ be a smooth elliptic curve
and $\mathcal{E}$ indecomposable rank $2$ vector bundle over $Z$
with $\deg(\mathcal{E}) = 0$ (see \cite{atiyah}). Put $S =
\mathbb{P}_{Z}(\mathcal{E})$ and let $C$ be the tautological
section on $S$. Then we have $C^{2} = 0$ and $K_{S} = -2C$. Let
$F$ be a fibre of the $\mathbb{P}^1$-bundle $S \to \mathbb{P}^1$.
Then the cone $\overline{NE}(S)$ is generated by two rays $R_{1} =
\mathbb{R}_{\scriptscriptstyle\geqslant 0}[C]$, $R_{2} =
\mathbb{R}_{\scriptscriptstyle\geqslant 0}[F]$, and there is no
curve $C' \ne C$ with $[C'] \in R_{1}$ (see \cite[Example
1.1]{shok-complements}). In particular, the linear system
$|-nK_S|$ is not free for any $n$. Consider the cone $X$ over $S
\subset \mathbb{P}^N$ with respect to some projective embedding
and the blow up $\sigma: Y \longrightarrow X$ of the vertex on $X$
with exceptional divisor $E$. We have $-(K_X + E) =
\sigma^{*}(\mathcal{O}_{X}(1)) + \pi^{*}(-K_{S})$, where $\pi: Y
\to S$ is the natural projection, which implies that $-(K_X + E)$
is nef and big. On the other hand, $(X,E)$ is purely log terminal
and $|-n(K_X + E)| \big\vert_{E} = |-nK_E|$ is not free for any
$n$ because $E \simeq S$. Moreover, $(X, E + \pi^{*}(C))$ is log
canonical, $-(K_X + E + \pi^{*}(C)) =
\sigma^{*}(\mathcal{O}_{X}(1)) + \pi^{*}(C)$ is nef and big, but
again $|-n(K_X + E + \pi^{*}(C))| \big\vert_{E} = |(-n/2)K_{E}|$
is not free. The latter shows that
Conjecture~\ref{theorem:reid-conjecture} is not true also for
strictly log canonical pairs (the special case of $\llcorner D
\lrcorner = 0$ and $\dim X \leqslant 4$ was treated in
\cite{gongyo}).
\end{example}

It follows from Example~\ref{example:counter-example} that the
case when $D$ contains a reduced part is far from being trivial.
The present paper aims to correct the main result of \cite{karz}.
We shall consider in some sense the simplest case when
Conjecture~\ref{theorem:reid-conjecture} is true for $\dim X
\geqslant 3$ and $\llcorner D \lrcorner \ne 0$:

\begin{theorem}
\label{theorem:main} Let $(X, D)$ be as above. Suppose that

\begin{itemize}

\item $X$ is a smooth $3$-fold and $ D = S$ is a smooth surface;

\item $S \cdot Z
> 0$ for every curve $Z$ on $S$ with $K_S \cdot Z = 0$.

\end{itemize}
Then the linear system $|-n(K_{X}+D)|$ is free for $n \gg 0$.
\end{theorem}

It follows from Example~\ref{example:counter-example} that the
assertion of Theorem ~\ref{theorem:main} is false without the
additional assumption on $K_S$-trivial curves. This suggests the
following
\begin{conjecture}
\label{theorem:my-conjecture} Let $(X, D)$ be as above. Suppose
that $X$ is $\mathbb{Q}$-factorial and $(K_{X} + D) \cdot S^{\dim
X - 1} \geqslant 0$ for every irreducible component $S \subseteq
\llcorner D \lrcorner$. Then the linear system $|-n(K_{X}+D)|$ is
free for $n \gg 0$.
\end{conjecture}

I would like to thank Y. Gongyo for pointing out the mistake in
\cite{karz}.

\section{Preliminaries}
\label{section:preliminary-results}

We use standard notation, notions and facts from the theory of
minimal models and singularities of pairs (see
\cite{kawamata-matsuda-matsuki}, \cite{kollar-mori},
\cite{kollar-sing-of-pairs}, \cite{kollar-92}). In what follows,
$(X,S)$ is the pair from Theorem~\ref{theorem:main}. In order to
prove Theorem~\ref{theorem:main}, we assume that
$\mathrm{Bs}(|-n(K_{X}+S)|) \ne
\emptyset$,\footnote{$\mathrm{Bs}(\mathcal{M})$ denotes the base
locus of the linear system $\mathcal{M}$.} where $n \gg 0$.

\begin{proposition}
\label{theorem:restriction-on-s-of-l} We have
$$
\mathrm{Bs}(|-n(K_{X}+S)|) \cap S =
\mathrm{Bs}(|-n(K_{X}+S)\big\vert_{S}|) = \mathrm{Bs}(|-nK_{S}|)
\ne \emptyset.
$$
\end{proposition}

\begin{proof}
Consider the exact sequence
\begin{eqnarray}
\nonumber 0 \to \mathcal{O}_{X}(-n(K_{X}+S)-S) \to
\mathcal{O}_{X}(-n(K_{X}+S)) \to
\mathcal{O}_{S}(-n(K_{X}+S)\big\vert_{S}) \to 0.
\end{eqnarray}
We have
\begin{eqnarray}
\nonumber H^{1}(X,\mathcal{O}_{X}(-n(K_{X}+S)-S)) =
H^{1}(X,\mathcal{O}_{X}(K_{X} - (n+1)(K_{X}+S))) = 0
\end{eqnarray}
by Kawamata--Viehweg Vanishing Theorem. This gives the exact
sequence
\begin{equation}
\nonumber H^{0}(X,\mathcal{O}_{X}(-n(K_{X}+S)) \to
H^{0}(S,\mathcal{O}_{S}(-n(K_{X}+S)\big\vert_{S})) \to 0,
\end{equation}
which implies that
$$
\mathrm{Bs}(|-n(K_{X}+S)|) \cap S =
\mathrm{Bs}(|-n(K_{X}+S)\big\vert_{S}|) = \mathrm{Bs}(|-nK_{S}|).
$$
Finally, if $\mathrm{Bs}(|-n(K_{X}+S)|) \cap S = \emptyset$, then
$\mathrm{Bs}(|-n(K_{X}+S)|) = \emptyset$ (see the proof of the
Basepoint-free Theorem in \cite{kollar-mori})), a contradiction.
\end{proof}

From Proposition~\ref{theorem:restriction-on-s-of-l} we get the
following

\begin{corollary}
\label{theorem:K-S-square} Equality $K_{S}^2 = 0$ holds.
\end{corollary}

\begin{proof}
Since $-(K_X + S)$ is nef, we have
$$
K_{S}^2 = (K_X + S)^2 \cdot S \geqslant 0
$$
(see \cite[Theorem 1.38]{kollar-mori}). Now, if $K_{S}^2 > 0$,
then $-K_S$ is nef and big, and the Basepoint-free Theorem implies
that $\mathrm{Bs}(|-nK_{S}|) = \emptyset$, a contradiction.
\end{proof}

\section{Reduction to the non-complementary case}
\label{section:complements}

We use notation and conventions of
Section~\ref{section:preliminary-results}. Let us show that the
surface $S$ does not have $\mathbb{Q}$-complements. Assume the
contrary. Then we have the following

\begin{lemma}
\label{theorem:S-is-rational} $S$ is a rational surface.
\end{lemma}

\begin{proof}
Suppose that $S$ is non-rational. Then it follows from the proof
of \cite[Theorem 1.3]{fed-kud}, \cite[Corollary
2.2]{shok-complements} and \cite[Example 2.1]{shok-complements}
that $\mathrm{Bs}(|-nK_{S}|) = \emptyset$, which contradicts
Proposition~\ref{theorem:restriction-on-s-of-l}.
\end{proof}

Thus, there exists a birational contraction $\chi: S
\longrightarrow \widetilde{S}$, where either $\chi$ is the blow up
of $\widetilde{S} = \mathbb{P}^{2}$ at some points $p_{1}, \ldots,
p_{9}$, or $\chi$ is the blow up of $\widetilde{S} =
\mathbb{F}_{m}$, $m \in \mathbb{N}$, at some points $q_{1},
\ldots, q_{8}$ (see Corollary~\ref{theorem:K-S-square}). To
simplify the notation, in what follows we assume that all $p_{i}$
(respectively, all $q_{i}$) are distinct. Further, by our
assumption the equivalence $K_{S} + \Delta \sim 0$ holds for some
effective $\mathbb{Q}$-divisor $\Delta$ such that the pair $(S,
\Delta)$ is log canonical. Then we have
$$
-K_{S} \sim \sum_{i = 1}^{N}\Delta_{i},
$$
$N \in \mathbb{N}$, where $\Delta_{i}$ are reduced and irreducible
curves such that $\Delta_i \cap \Delta_j \ne \emptyset$ for all $i
\ne j$ and the intersection is transversal.

\begin{lemma}
\label{theorem:curves-on-p-2} Equality $h^{0}(S,
\mathcal{O}_{S}(-nK_{S})) = 1$ holds.
\end{lemma}

\begin{proof}
We have $h^{0}(S, \mathcal{O}_{S}(-nK_{S})) > 0$. Suppose that
$h^{0}(S, \mathcal{O}_{S}(-nK_{S})) \geqslant 2$. Then, since the
sum $\sum_{i = 1}^{N}\Delta_{i}$ is connected and $-K_S$ is nef
with $K_{S}^{2} = 0$, $|-nK_{S}|$ is a free pencil, which
contradicts Proposition~\ref{theorem:restriction-on-s-of-l}.
\end{proof}

\begin{proposition}
\label{theorem:p-2-case} If $N = 1$, then $\widetilde{S} \ne
\mathbb{P}^{2}$.
\end{proposition}

\begin{proof}
Suppose that $\widetilde{S} = \mathbb{P}^{2}$. Let us consider two
cases:

\smallskip

{\bf Case (1).} The curve $C = \Delta_{1}$ is smooth. Write
$$
S\big\vert_{S} = \chi^{*}(aL) + \sum_{i=1}^{9}a_{i}E_{i},
$$
where $L$ is a line on $\mathbb{P}^{2}$, $a$ and $a_{i} \in
\mathbb{Z}$, $E_i = \chi^{-1}(p_{i})$. Let $\varphi: Y
\longrightarrow X$ be the blow up of $X$ at $C$ with exceptional
divisor $E$. For $S_{Y} = \varphi_{*}^{-1}(S)$, $\varphi$ induces
an isomorphism $\varphi_{\scriptscriptstyle S}: S_{Y} \simeq S$
such that $\varphi_{\scriptscriptstyle S}(S_{Y} \cap E) = C$ and
$\varphi_{\scriptscriptstyle S}$ is identical out of $C_{Y} =
S_{Y} \cdot E$, which implies that $\varphi_{\scriptscriptstyle
S}$ is an automorphism of $S$, identical on $\mathrm{Pic}(S)$. In
particular, we have
$$
E \cdot C_Y = E \cdot E \cdot S_Y = C_{Y}^2 = 0,
$$
which together with equalities
$$
K_{Y} + S_Y = \varphi^{*}(K_{X} + S) \qquad\mbox{and}\qquad S_Y =
\varphi^{*}(S) - E
$$
implies that to prove Proposition~\ref{theorem:p-2-case} we may
pass from $(X,S)$ to the pair $(Y, S_{Y})$. Moreover, we have
$$
S_{Y}\big\vert_{S_{Y}} = \chi^{*}(a_{Y}L_{Y}) +
\sum_{i=1}^{9}a_{i,Y}E_{i,Y},
$$
where $L_{Y} = \varphi_{*}^{-1}(L)$, $E_{i,Y} =
\varphi_{*}^{-1}(E_{i})$, $a_{Y}$ and $a_{i,Y} \in \mathbb{Z}$,
which implies that
$$
-a_{i,Y} = S_{Y} \cdot E_{i,Y} = S \cdot E_{i} - E \cdot E_{i,Y}
\leqslant S \cdot E_{i} - 1 = -a_{i} - 1,
$$
and hence $a_{i,Y} > a_{i}$.

Thus, applying the above arguments to $(Y,S_{Y})$, after a number
of blow ups we obtain that to prove
Proposition~\ref{theorem:p-2-case} we may assume that $a_{i}
> 0$ for all $i$. In particular, we have
$$
(K_{X} + S) \cdot E_{1} = K_S \cdot E_{1} < 0
\qquad\mbox{and}\qquad S\cdot E_{1} = -a_{1} < 0,
$$
and it follows from the Cone Theorem that equality $E_1 \equiv
\sum_{i}R_i + \sum_{j}C_{j}$ holds on $X$, where $R_i$ are $(K_{X}
+ S)$-negative extremal rays and $(K_{X} + S) \cdot C_j = 0$ for
all $j$. Moreover, by assumption on $K_S$-trivial curves we have
$S \cdot C_j \geqslant 0$ for all $j$, which implies that there
exists a $(K_{X} + S)$-negative extremal ray $R$ on $X$ such that
$S \cdot R < 0$. In particular, we have $R \subset S$ and the
extremal contraction $\mathrm{cont_{\scriptscriptstyle R}}: X
\longrightarrow W$ is birational.

\begin{lemma}
\label{theorem:div-contr} $\mathrm{cont_{\scriptscriptstyle R}}$
is not a divisorial morphism.
\end{lemma}

\begin{proof}
Assume the contrary. Then the image of $S$ is either a point or a
curve. But the first case is impossible because $(K_{X} + S) \cdot
C = 0$. Thus, $\mathrm{cont_{\scriptscriptstyle R}}(S)$ is a
curve. Then there exists a birational contraction $\chi': S
\longrightarrow \mathbb{P}^{2}$, which is the blow up at some
points $p'_{1}, \ldots, p'_{9}$ on $\mathbb{P}^{2}$, with
exceptional curves $E'_{1}, \ldots, E'_{9}$ such that

\begin{itemize}

\item $E'_{1} \cdot R = 1$ and $E'_{1} \cdot Z = 0$
for some curve $Z$ on $S$ such that $R =
\mathbb{R}_{\scriptscriptstyle \geqslant 0}[Z]$;

\item $R = \mathbb{R}_{\scriptscriptstyle \geqslant 0}[E'_{i}]$ for
all $i \geqslant 2$.

\end{itemize}
Let $\varphi: Y \longrightarrow X$ be the blow up of $X$ at
$E'_{1}$ with exceptional divisor $E$. Then $Y$ possesses a
$(K_{Y} + S_{Y})$-negative extremal ray $R_Y = \varphi_{*}^{-1}(R)
+ \alpha e$, where $\alpha \in \mathbb{Q}$, $S_Y =
\varphi_{*}^{-1}(S)$ and $e$ is the numerical class of a fibre of
$\varphi$. Note that $\alpha \leqslant 0$, which implies that
$$
S_Y \cdot R_Y = S \cdot R + \alpha < 0
$$
and hence $R_Y \subset S_Y$. In particular, $R_Y =
\varphi_{*}^{-1}(R)$. Moreover, for the curves $E'_{1,Y} = S_{Y}
\cdot E$, $Z_Y = \varphi_{*}^{-1}(Z)$ and $E'_{i,Y} =
\varphi_{*}^{-1}(E'_{i})$, $i \geqslant 2$, we have

\begin{itemize}

\item $E'_{1,Y} \cdot R_{Y} = 1$ and $E'_{1,Y} \cdot Z_{Y} = 0$;

\item $R_{Y} = \mathbb{R}_{\scriptscriptstyle \geqslant 0}[E'_{i,Y}] = \mathbb{R}_{\scriptscriptstyle \geqslant 0}[Z_{Y}]$ for
all $i \geqslant 2$.

\end{itemize}
On the other hand, we get
$$
0 = E'_{1,Y} \cdot Z_Y = E \cdot Z_Y = E \cdot R_Y = E'_{1,Y}
\cdot R_Y = 1,
$$
a contradiction.
\end{proof}

Thus, $\mathrm{cont_{\scriptscriptstyle R}}$ is a small
contraction. Then $R$ is generated by a $(-1)$-curve on $S$.
Consider the $(K_{X} + S)$-flip:
$$
\xymatrix{
&X\ar@{-->}[rr]^{\tau}\ar@{->}[rd]_{\mathrm{cont_{\scriptscriptstyle R}}}&&X^{+}\ar@{->}[ld]^{\mathrm{cont_{\scriptscriptstyle R}^{+}}}&\\
&&W,&&}
$$
so that the map $\tau$ is an isomorphism in codimension $1$, for
every curve $R^{+} \subset X^{+}$, which is contracted by
$\mathrm{cont_{\scriptscriptstyle R}^{+}}$, we have $(K_{X^{+}} +
S^{+}) \cdot R^{+}
> 0$, where $S^{+} = \tau_{*}(S)$,
$3$-fold $X^{+}$ is $\mathbb{Q}$-factorial and the pair $(X^{+},
S^{+})$ is purely log terminal (see \cite{kollar-92} and
\cite[Proposition 3.36, Lemma 3.38]{kollar-mori}). Let
$$
\xymatrix{
&&T\ar@{->}[ld]_{f}\ar@{->}[rd]^{f^{+}}&&\\%
&X\ar@{-->}[rr]_{\tau}&&X^{+}&}
$$
be resolution of indeterminacies of $\tau$ over $W$. Then $f$ is a
sequence of blow ups at smooth centers over $R$ with exceptional
divisors $G_{1}, \ldots, G_{s} \subset T$ such that $G_{i}$
constitute the $f^{+}$-exceptional locus and $Z = f^{+}(\sum_{i =
1}^{s} G_{i})$ is a union of all $\mathrm{cont_{\scriptscriptstyle
R}^{+}}$-exceptional curves.

\begin{lemma}
\label{theorem:tau-induces-good-contraction-on-S} We have $Z
\subseteq \mathrm{Bs}(|-n(K_{X^{+}} + D^{+})|)$ and $R^{+}
\not\subset S^{+}$ for every $R^{+} \subseteq Z$.
\end{lemma}

\begin{proof}
The statement follows from conditions $K_{X^{+}} + S^{+} =
\tau_{*}(K_{X} + S)$, $R \not\subset \mathrm{Supp}(-K_{S})$,
$(K_{X^{+}} + S^{+}) \cdot R^{+} > 0$ for every $R^{+} \subseteq
Z$ and the fact that $f_{*}^{-1}(S) \cdot f_{*}^{-1}(-n(K_{X} +
S)) = f_{*}^{-1}(S \cdot (-n(K_{X} + S)))$ (the latter holds
because $f$ is a sequence of blow ups at smooth centers).
\end{proof}

It follows from
Lemma~\ref{theorem:tau-induces-good-contraction-on-S} that $S^{+}
\simeq \mathrm{cont_{\scriptscriptstyle R}}(S)$ and $\tau$ induces
contraction $\tau_{\scriptscriptstyle S}: S \longrightarrow S^{+}$
of the $(-1)$-curve in $R$. Then, since $K_{X^{+}} + S^{+} =
\tau_{*}(K_{X} + S)$, the divisor
$$
-(K_{X^{+}} + S^{+})\big\vert_{S^{+}} \equiv -K_{S^{+}} =
\tau_{\scriptscriptstyle S*}(C)
$$
is nef and big. Moreover, the surface $S^{+}$ has only log
terminal singularities by the Inversion of adjunction, which
implies that $\mathrm{Bs}(|-nK_{S^{+}}|) = \emptyset$ by the
Basepoint-free Theorem.

\begin{lemma}
\label{theorem:many-sections} Inequality $h^{0}(S,
\mathcal{O}_{S}(-nK_{S})) \geqslant 2$ holds.
\end{lemma}

\begin{proof}
We have $R^{1}(\mathrm{cont_{\scriptscriptstyle
R}})_{*}(-n(K_{X}+S)-S) = 0$ by the relative Kawamata--Viehweg
Vanishing Theorem. This and the isomorphism $S^{+} \simeq S^{*} =
\mathrm{cont_{\scriptscriptstyle R}}(S)$ imply that the
push-forwards to $W$ of exact sequences $0 \to
\mathcal{O}_{X}(-n(K_{X}+S)-S) \to \mathcal{O}_{X}(-n(K_{X}+S))
\to \mathcal{O}_{S}(-n(K_{X}+S)\big\vert_{S}) \to 0$ and $0 \to
\mathcal{O}_{X^{+}}(-n(K_{X^{+}}+S^{+})-S^{+}) \to
\mathcal{O}_{X^{+}}(-n(K_{X^{+}}+S^{+})) \to
\mathcal{O}_{S^{+}}(-n(K_{X^{+}}+S^{+})\big\vert_{S^{+}}) \to 0$
coincide with exact sequence
$$
0 \to \mathcal{O}_{W}(-n(K_{W}+S^{*})-S^{*}) \to
\mathcal{O}_{W}(-n(K_{W}+S^{*})) \to
\mathcal{O}_{S^{*}}(-n(K_{W}+S^{*})\big\vert_{S^{*}}) \to 0.
$$
Then it follows from $\mathrm{Bs}(|-nK_{S^{+}}|) = \emptyset$ that
$h^{0}(S, \mathcal{O}_{S}(-nK_{S})) \geqslant 2$.
\end{proof}

From Lemma~\ref{theorem:many-sections} we get contradiction with
Lemma~\ref{theorem:curves-on-p-2}. Thus, {\bf Case (1)} is
impossible, and we pass to

\smallskip

{\bf Case (2).} The curve $C = \Delta_{1}$ is singular. Since $C
\sim -K_{S}$ and the pair $(S, C)$ is log canonical, we have
$p_{a}(C) = 1$ and the only singular point on $C$ is an ordinary
double point $O$. Let $\varphi: Y \longrightarrow X$ be the blow
up of $X$ at $C$ with exceptional divisor $E$. Locally near $O$
there is an analytic isomorphism
$$
(X,S,\Delta) \simeq \big(\mathbb{C}^{3}_{x,y,x}, \{x = 0\}, \{yz =
0\}\big).
$$
Then locally over $O$ we have the following representation for
$Y$:
$$
Y = \{yzt_{0} = xt_{1}\} \subset \mathbb{C}^{3}_{x,y,z} \times
\mathbb{P}^{1}_{t_{0},t_{1}},
$$
which implies that the only singular point on $Y$ is a
non-$\mathbb{Q}$-factorial quadratic singularity. Then, since
$$
K_{Y} + \varphi_{*}^{-1}(S) = \varphi^{*}(K_{X} + S),
$$
after a small resolution $\psi: \widetilde{Y} \longrightarrow Y$
we may pass from $(X,S)$ to the pair
$(\widetilde{Y},\psi_{*}^{-1}(\varphi_{*}^{-1}(S)))$ as above and
apply the arguments from {\bf Case (1)}.
\end{proof}

Applying the same arguments as in the proof of
Proposition~\ref{theorem:curves-on-p-2} to $\widetilde{S} =
\mathbb{F}_{m}$, we see that the case $N = 1$ is impossible.
Finally, in the case when $N \geqslant 2$ we apply the same
arguments as in the proof of
Proposition~\ref{theorem:curves-on-p-2}, replacing the curve
$\Delta_1$ with the cycle $\sum_{i = 1}^{N}\Delta_{i}$.

Thus, the surface $S$ does not have $\mathbb{Q}$-complements, and
we get the following

\begin{corollary}
\label{theorem:every-thing-is-smooth} In the notation of
Example~\ref{example:counter-example}, we have $S =
\mathbb{P}_{Z}(\mathcal{E})$ and
$\mathrm{Supp}(-n(K_{X}+S)\big\vert_{S}) = C$. In particular,
$\mathrm{Bs}(|-n(K_{X}+S)|) \cap S = C$.
\end{corollary}

\begin{proof}
The statement follows from  \cite[Theorem 1.3]{fed-kud} and
Proposition~\ref{theorem:restriction-on-s-of-l}.
\end{proof}

Let $F$ be a fibre of the $\mathbb{P}^1$-bundle $S \to
\mathbb{P}^1$. Write
\begin{equation}
\nonumber S \big \vert _{S} = aC - bF,
\end{equation}
where $a$, $b \in \mathbb{Z}$. Note that $b < 0$ because $C$ is
$K_S$-trivial and hence $0 < S \cdot C = S \big\vert_{S} \cdot C =
-b$ by assumption on $K_S$-trivial curves.

\begin{lemma}
\label{theorem:mori-cone-of-b} Equality $\deg(\mathcal{N}_{C/X}) =
-b$ holds.
\end{lemma}

\begin{proof}
Since $C$ is a smooth elliptic curve, we have
$$
\deg(\mathcal{N}_{C/X}) = -K_{X} \cdot C = ((2 + a)C - bF) \cdot C
= -b.
$$
\end{proof}

Put $\mathcal{L}_{n} = |-n(K_{X} + S)|$. Then for the general
element $L_{n} \in \mathcal{L}_{n}$ we have
$$
L_{n} = M + \sum r_{i,S}B_{i,S} + \sum r_{i}B_{i},
$$
where $B_{i}$, $B_{i,S}$ are the base components of
$\mathcal{L}_{n}$, $r_{i}$, $r_{i,S} \geqslant 0$ the
corresponding multiplicities, $B_{i} \cap S = \emptyset$, $B_{i,S}
\cap S \ne \emptyset$ for all $i$, and the linear system $|M|$ is
movable. According to
Corollary~\ref{theorem:every-thing-is-smooth}, we have
$\mathrm{Bs}(|-n(K_{X}+D)|) \cap S = C$ and $B_{i,S} \cap S = C$
for all $i$, which implies that $\mathrm{Bs}(|M|) \cap S = C$ or
$\emptyset$. In what follows, we assume that $\mathrm{Bs}(|M|) =
\mathrm{Bs}(|M|) \cap S$ (see the proof of the Basepoint-free
Theorem in \cite{kollar-mori}). Furthermore, arguing exactly as in
the proof of Proposition~\ref{theorem:curves-on-p-2}, we can
replace $X$ by its blow up at the curve $C$. Then, after applying
Corollary~\ref{theorem:every-thing-is-smooth} and a number of blow
ups, in what follows we assume that the following conditions are
satisfied:

\begin{itemize}

\item $r_{i,S} = r > 0$ and $B_{i,S} = B$ for all $i$, where $B =
\mathbb{P}_{C}(\mathcal{N}_{C/X})$ with $B^{3} =
-\deg(\mathcal{N}_{C/X}) = b$ (see
Lemma~\ref{theorem:mori-cone-of-b});

\item $S \cdot B = C$;

\item the linear system $|M|$ is free and $M \cap B = \emptyset$;

\item $B_{j} \cap B \ne \emptyset$ for exactly one $j$ and the intersection is transversal, $r_{j} = r$,
$B_{j}^{2} \cdot B = b$.

\end{itemize}

\section{Exclusion of the non-complementary case}
\label{section:special-case}

We use notation and conventions of
Section~\ref{section:complements}. Let $\varphi: Y \longrightarrow
X$ be the blow up of $X$ at the curve $C$ with exceptional divisor
$E$. Put
$$
\begin{array}{c}
\nonumber S_{Y} = \varphi_{*}^{-1}(S), \qquad B_{Y} =
\varphi_{*}^{-1}(B), \qquad M_{Y} = \varphi_{*}^{-1}(M), \qquad
B_{i,Y} = \varphi_{*}^{-1}(B_{i}).
\end{array}
$$
Then for $m \gg 0$, $0 < \delta_{1}, \ \delta_{2} \ll 1$ and $0 <
c \leqslant 1$ we write
\begin{eqnarray}
\label{r-test-div}R = \varphi^{*}(-(K_{X} + S) +
mL_{n} - cL_{n}) + cM_{Y} + \delta_{1}S_{Y} + \delta_{2}E = \\
\nonumber = \varphi^{*}(mL_{n}) + (-1 + \delta_{1})S_{Y} +
(\delta_{2} - cr)E - \\
\nonumber - crB_{Y} - crB_{j,Y} - \sum_{i \ne j} cr_{i}B_{i,Y} -
K_{Y}.
\end{eqnarray}

\begin{proposition}
\label{theorem:r-is-nef-big} The divisor $R$ is nef and big for
$\delta_{1} \geqslant \delta_{2}$.
\end{proposition}

\begin{proof}
Since the divisors $-(K_{X} + S)$ and $M_{Y}$ are nef and big, it
suffices to prove that the divisor
$$
R = \varphi^{*}(-(K_{X} + S) + mL_{n} - cL_{n}) + cM_{Y} +
\delta_{1}S_{Y} + \delta_{2}E
$$
intersects every curve on the surfaces $S_{Y}$ and $E$
non-negatively.

\begin{lemma}
\label{theorem:case-of-s-y} Inequality $R \cdot Z \geqslant 0$
holds for every curve $Z$ on $S_{Y}$.
\end{lemma}

\begin{proof}
Since $S_Y \simeq S$, the cone $\overline{NE}(S_{Y})$ is generated
by the classes $[C_{Y}] = [S_{Y} \cdot E]$ and $[F_{Y}] =
[\varphi_{*}^{-1}(F)]$. Thus, it is enough to consider only the
cases when $Z = C$ or $F$.

We have $E \cdot C_{Y} = 0$ and
$$
S_{Y} \cdot C_{Y} = S \cdot C = -b > 0,
$$
which implies that $R \cdot C_{Y} > 0$. Furthermore, we have
$$
R \cdot F_{Y} \gg \varphi^{*}(L_{n}) \cdot F_{Y} = L_{n} \cdot F
\geqslant nC \cdot F = n \gg 0,
$$
and the assertion follows.
\end{proof}

\begin{lemma}
\label{theorem:case-of-e} Inequality $R \cdot Z \geqslant 0$ holds
for every curve $Z$ on $E$ and $\delta_{1} \geqslant \delta_{2}$.
\end{lemma}

\begin{proof}
Let $F_{E}$ be a fibre of the $\mathbb{P}^{1}$-bundle $E =
\mathbb{P}_{C}(\mathcal{N}_{C/X})$. We have
$$
(B_{Y}\big\vert_{E})^{2} = (\varphi^{*}(B) - E)^{2} \cdot E = 2B
\cdot C + E^{3} = 2C^{2} + E^{3} = -\deg(\mathcal{N}_{C/X}) = b <
0,
$$
which implies that the cone $\overline{NE}(E)$ is generated by the
classes $[-E\big\vert_{E}] = [B_{Y}\big\vert_{E}]$ and $[F_{E}]$
(see \cite[Lemma 1.22]{kollar-mori}). Thus, it is enough to
consider only the cases when $Z = -E\big\vert_{E}$ or $F_{E}$.

We have
$$
S_{Y} \cdot (-E\big\vert_{E}) = -S_{Y} \cdot E^{2} = 0,
$$
which implies that
$$
R \cdot (-E\big\vert_{E}) \geqslant \delta_{2}E \cdot
(-E\big\vert_{E}) = -b\delta_{2}
> 0.
$$
Furthermore, we have
$$
S_{Y} \cdot F_{E} = 1 \qquad\mbox{and}\qquad E \cdot F_{E} = -1,
$$
which implies that
$$
R \cdot F_{E} = \delta_{1} - \delta_{2} \geqslant 0,
$$
and the assertion follows.
\end{proof}

Lemmas~\ref{theorem:case-of-s-y} and \ref{theorem:case-of-e} prove
Proposition~\ref{theorem:r-is-nef-big}.
\end{proof}

Take $c = 1/r$ in \eqref{r-test-div}. Then we obtain
$$
\ulcorner R \urcorner = \varphi^{*}(mL_{n}) - B_{Y} - B_{j,Y} +
\sum_{i \ne j} \ulcorner -cr_{i} \urcorner B_{i,Y} - K_{Y},
$$
and Proposition~\ref{theorem:r-is-nef-big} and Kawamata--Viehweg
Vanishing Theorem imply that
\begin{equation}
\label{coh-ies-1} H^{i}(Y,\mathcal{O}_{Y}(\varphi^{*}(mL_{n}) -
B_{Y} - B_{j,Y} + \sum_{i \ne j} \ulcorner -cr_{i} \urcorner
B_{i,Y}))=0
\end{equation}
for all $i > 0$.

\begin{lemma}
\label{theorem:ex-sec-2} Inequality
$$
H^{0}(B_{Y}, \mathcal{O}_{B_{Y}}((\varphi^{*}(mL_{n}) - B_{j,Y}
+\sum_{i \ne j} \ulcorner -cr_{i} \urcorner
B_{i,Y})\big\vert_{B_{Y}})) \ne 0
$$
holds.
\end{lemma}

\begin{proof}
Note that $(\sum_{i \ne j} \ulcorner -cr_{i} \urcorner
B_{i,Y})\big\vert_{B_{Y}} = 0$. Let us prove that
$$
H^{0}(B_{Y}, \mathcal{O}_{B_{Y}}((\varphi^{*}(mL_{n})-
B_{j,Y})\big\vert_{B_{Y}})) \ne 0.
$$
We have
$$
\varphi^{*}(mL_{n}) = mM_{Y} + mrB_{Y} + mrB_{j,Y} + mrE + \sum_{i
\ne j} mr_{i} B_{i,Y},
$$
which implies that
$$
\varphi^{*}(mL_{n})\big\vert_{B_{Y}} = mrB_{Y}\big\vert_{B_{Y}} +
mrB_{j,Y}\big\vert_{B_{Y}} + mrE\big\vert_{B_{Y}}.
$$
Further, since $B_{Y} = \varphi^{*}(B) - E$ and $\varphi^{*}(B)
\cdot E^{2} = -B \cdot C = 0$, we obtain
$$
(E\big\vert_{B_{Y}})^{2} = E^{2} \cdot B_{Y} = -E^{3} = -b
$$
and
$$
(B_{j,Y}\big\vert_{B_{Y}})^{2} = B_{j}^{2} \cdot B = b,
$$
which implies that $E\big\vert_{B_{Y}}$ is the tautological
section of the $\mathbb{P}^{1}$-bundle $B_{Y} =
\mathbb{P}_{C}(\mathcal{N}_{C/X})$ with a fibre $F_{B_{Y}}$, and
$B_{j,Y}\big\vert_{B_{Y}} \sim E\big\vert_{B_{Y}} + bF_{B_{Y}}$.
Furthermore, we have
$$
(B_{Y}\big\vert_{B_{Y}})^{2} = B_{Y}^{3} = \varphi^{*}(B)^{3} -
E^{3} = 0
$$
and
$$
B_{Y}\big\vert_{B_{Y}} \cdot E \big\vert_{B_{Y}} = B_{Y}^{2} \cdot
E = E^{3} = b,
$$
which implies that $B_{Y}\big\vert_{B_{Y}} \sim bF_{B_{Y}}$. Thus,
we get
$$
\varphi^{*}(mL_{n})\big\vert_{B_{Y}} \sim
2mrB_{j,Y}\big\vert_{B_{Y}},
$$
which implies that
$$
\varphi^{*}(mL_{n})\big\vert_{B_{Y}} - B_{j,Y}\big\vert_{B_{Y}}
\sim (2mr-1)B_{j,Y}\big\vert_{B_{Y}}
$$
and hence $H^{0}(B_{Y}, \mathcal{O}_{B_{Y}}((\varphi^{*}(mL_{n})-
B_{j,Y})\big\vert_{B_{Y}})) \ne 0$.
\end{proof}

From \eqref{coh-ies-1} and the exact sequence
\begin{eqnarray}
\nonumber 0 \to \mathcal{O}_{Y}(\varphi^{*}(mL_{n}) - B_{Y} -
B_{j,Y} + \sum_{i \ne j} \ulcorner -cr_{i} \urcorner B_{i,Y})
\to \\
\nonumber \to \mathcal{O}_{Y}(\varphi^{*}(mL_{n}) - B_{j,Y} +
\sum_{i \ne j} \ulcorner -cr_{i} \urcorner B_{i,Y}) \to \\
\nonumber \to \mathcal{O}_{B_{Y}}((\varphi^{*}(mL_{n}) - B_{j,Y}
+\sum_{i \ne j} \ulcorner -cr_{i} \urcorner
B_{i,Y})\big\vert_{B_{Y}}) \to 0
\end{eqnarray}
we get exact sequence
\begin{eqnarray}
\nonumber 0 \to H^{0}(Y,\mathcal{O}_{Y}(\varphi^{*}(mL_{n}) -
B_{Y}- B_{j,Y} +
\sum_{i \ne j} \ulcorner -cr_{i} \urcorner B_{i,Y})) \to \\
\nonumber \to H^{0}(Y,\mathcal{O}_{Y}(\varphi^{*}(mL_{n})-
B_{j,Y}+\sum_{i \ne j} \ulcorner -cr_{i}
\urcorner B_{i,Y})) \to \\
\nonumber \to H^{0}(B_{Y},
\mathcal{O}_{B_{Y}}((\varphi^{*}(mL_{n})- B_{j,Y}+\sum_{i \ne j}
\ulcorner -cr_{i} \urcorner B_{i,Y})\big\vert_{B_{Y}})) \to 0,
\end{eqnarray}
which implies, since $-r_{i} \leqslant \ulcorner -cr_{i} \urcorner
\leqslant 0$ and $B_{Y}$, $B_{j,Y}$, $B_{i,Y}$ are the base
components of the linear system $|\varphi^{*}(mL_{n})|$, that
\begin{eqnarray}
\nonumber H^{0}(Y,\mathcal{O}_{Y}(\varphi^{*}(mL_{n}) - B_{Y}-
B_{j,Y} + \sum_{i \ne j} \ulcorner -cr_{i} \urcorner
B_{i,Y})) \simeq \\
\nonumber \simeq H^{0}(Y,\mathcal{O}_{Y}(\varphi^{*}(mL_{n}) -
B_{j,Y} + \sum_{i \ne j} \ulcorner -cr_{i} \urcorner B_{i,Y}))
\simeq H^{0}(Y,\mathcal{O}_{Y}(\varphi^{*}(mL_{n})))
\end{eqnarray}
and
$$
H^{0}(B_{Y}, \mathcal{O}_{B_{Y}}((\varphi^{*}(mL_{n})-
B_{j,Y}+\sum_{i \ne j} \ulcorner -cr_{i} \urcorner
B_{i,Y})\big\vert_{B_{Y}})) = 0,
$$
a contradiction with Lemma~\ref{theorem:ex-sec-2}.
Theorem~\ref{theorem:main} is completely proved.

\end{document}